\documentclass[11pt]{article}
\usepackage[includeheadfoot,margin=1in]{geometry}
\usepackage{times}
\usepackage{cite,setspace}
\usepackage{amsmath,amsfonts,amssymb,amsthm}
\usepackage{mathtools}
\usepackage{enumerate}
\usepackage{verbatim}
\usepackage{commath}
    \usepackage[usenames]{color}
    \definecolor{plum}  {rgb}{.4,0,.4}
    \definecolor{BrickRed} {rgb}{0.6,0,0}
	\definecolor{DarkBlue} {rgb}{0,0,0.6}
\usepackage[plainpages=false,pdfpagelabels,colorlinks=true,linkcolor=BrickRed,citecolor=plum,urlcolor=DarkBlue]{hyperref}
\usepackage[mathcal]{euscript}
\usepackage[round]{natbib}
\usepackage{cleveref}

\def\ddefloop#1{\ifx\ddefloop#1\else\ddef{#1}\expandafter\ddefloop\fi}

\def\ddef#1{\expandafter\def\csname b#1\endcsname{\ensuremath{\boldsymbol{#1}}}}
\ddefloop ABCDEFGHIJKLMNOPQRSTUVWXYZabcdefghijklmnopqrstuvwxyz\ddefloop

\def\ddef#1{\expandafter\def\csname c#1\endcsname{\ensuremath{\mathcal{#1}}}}
\ddefloop ABCDEFGHIJKLMNOPQRSTUVWXYZ\ddefloop
 
\def\ddef#1{\expandafter\def\csname s#1\endcsname{\ensuremath{\mathsf{#1}}}}
\ddefloop ABCDEFGHIJKLMNOPQRSTUVWXYZ\ddefloop

\def\Reals{{\mathbb R}}

\def\Ex{{\mathbf E}} 

\makeatletter
\newsavebox{\@brx}
\newcommand{\llangle}[1][]{\savebox{\@brx}{\(\m@th{#1\langle}\)}%
  \mathopen{\copy\@brx\kern-0.5\wd\@brx\usebox{\@brx}}}
\newcommand{\rrangle}[1][]{\savebox{\@brx}{\(\m@th{#1\rangle}\)}%
  \mathclose{\copy\@brx\kern-0.5\wd\@brx\usebox{\@brx}}}
\makeatother

\newtheorem{theorem}{Theorem}

\newtheorem{lemma}{Lemma}
\newtheorem{corollary}{Corollary}
\newtheorem{remark}{Remark}

\begin{document}

\title{
Talagrand Meets Talagrand: Upper and Lower Bounds on Expected Soft Maxima of Gaussian Processes with Finite Index Sets}
\author{Yifeng Chu \\ \href{mailto:ychu26@illinois.edu}{ychu26@illinois.edu} \\ 
\and Maxim Raginsky \\ \href{mailto:maxim@illinois.edu}{maxim@illinois.edu} }
\date{}
\maketitle

\begin{abstract}%
Analysis of extremal behavior of stochastic processes is a key ingredient in a wide variety of applications, including probability, statistical physics, theoretical computer science, and learning theory.  In this paper, we consider centered Gaussian processes on finite index sets and investigate expected values of their smoothed, or ``soft,'' maxima. We obtain upper and lower bounds for these expected values using a combination of ideas from statistical physics (the Gibbs variational principle for the equilibrium free energy and replica-symmetric representations of Gibbs averages) and from probability theory (Sudakov minoration). These bounds are parametrized by an inverse temperature $\beta > 0$ and reduce to the usual Gaussian maximal inequalities in the zero-temperature limit $\beta \to \infty$. We provide an illustration of our methods in the context of the Random Energy Model, one of the simplest models of physical systems with random disorder.
\end{abstract}


\section{Introduction and informal summary of results}

The study of extremal behavior of stochastic processes is of fundamental importance in a wide variety of applications, including pure and applied probability, theoretical computer science, mathematical statistics, and statistical physics. For example, the method of generic chaining \citep{Talagrand_2014} gives sharp upper and lower bounds on the expected supremum of a centered Gaussian process $X = (X_t)_{t \in T}$ in terms of the geometry of the index set $T$ induced by the natural metric $d(s,t) := (\Ex|X_s - X_t|^2)^{1/2}$. The underlying intuition here is that, for finite $T$, $\sqrt{\log |T|}$ is a good proxy for the expected supremum of the centered Gaussian process $(X_t)_{t \in T}$, in the sense that
\begin{align}\label{eq:expected_max}
	a \sqrt{\log |T|} \lesssim \Ex\Big[\sup_{t \in T}X_t\Big] \lesssim \Delta\sqrt{\log |T|},
\end{align}
where $a = \min \{d(s,t) : s,t \in T, s \neq t\}$ is the minimum separation between the elements of $T$ and $\Delta = \max \{d(s,t) : s,t \in T\}$ is the diameter of $T$. Here, the upper bound is by the well-known maximal inequality for Gaussian random variables, while the lower bound follows from Sudakov minoration. Another important observation is that the zero-mean Gaussian random variables $X_s$ and $X_t$ are highly correlated when $d(s,t)$ is very small and nearly independent when $d(s,t)$ is very large.  The method of generic chaining is, essentially, a multiscale iterative application of these ideas, where at each scale $n=0,1,2,\dots$ one approximates the process $X$ by its ``subsampled'' version at a finite set of indices $T_n \subseteq T$ satisfying the cardinality bound $|T_n| < 2^{2^n}$. 

While the expected supremum is a natural measure of the size of $X$, there are various ``smoothed'' versions of it that may be of interest in their own right, for example in the context of concentration and superconcentration phenomena in Gaussian random fields \citep{Chatterjee_2014}. For example, if $T$ is finite, then $\max_{t \in T} X_t$ can be bounded from above and from below by the ``$\beta$-softmax'' for every $\beta > 0$:
\begin{align}\label{eq:softmax_bounds}
	\max_{t \in T} X_t \le \frac{1}{\beta}\log \sum_{t \in T}e^{\beta X_t} \le \max_{t \in T} X_t + \frac{\log |T|}{\beta}.
\end{align}
Another smooth proxy for the maximum is 
\begin{align}
	g(X; \beta) := \frac{1}{Z(\beta)}\sum_{t \in T}X_t e^{\beta X_t},
\end{align}
where  $Z(\beta) := \sum_{t \in T}e^{\beta X_t}$ is known as the partition function. It satisfies the inequalities
\begin{align}\label{eq:gibbs_bounds}
	\max_{t \in T}X_t - \frac{\log |T|}{\beta} \le g(X; \beta) \le \max_{t \in T} X_t;
\end{align}
moreover, as we increase the parameter $\beta > 0$, $g(\beta;X)$ varies from the sample average $\frac{1}{|T|}\sum_{t \in T}X_t$ at $\beta = 0$ to the maximum $\max_{t \in T} X_t$ in the limit as $\beta \to \infty$.

With these considerations in mind, let $X = (X_t)_{t \in T}$ be a centered Gaussian process with a finite index set $T$. We will assume throughout that $\Ex|X_s - X_t|^2 > 0$ for all $s \neq t$, i.e., $d$ is a metric and not a pseudometric. For each $\beta > 0$, we introduce the \textit{quenched Gibbs average}
\begin{align}\label{eq:quenched_Gibbs_energy}
	g(\beta) := \Ex\Bigg[\frac{1}{Z(\beta)}\sum_{t \in T}X_t e^{\beta X_t}\Bigg] = \Ex g(X; \beta)
\end{align}
and the \textit{quenched equilibrium free energy}
\begin{align}\label{eq:quenched_free_energy}
	\phi(\beta) := \frac{1}{\beta}\Ex\log\Bigg(\frac{1}{|T|}\sum_{t \in T}e^{\beta X_t}\Bigg) = \frac{1}{\beta} \Ex \log Z(\beta) - \frac{\log |T|}{\beta}.
\end{align}
The terminology originates in statistical physics of disordered systems \citep{Bovier2006}, where $T$ is interpreted as the state space of a physical system and $-X_t$ is the random  energy of the state $t \in T$. Each realization of $X$ gives rise to a family of random Gibbs measures $(\nu_\beta)_{\beta > 0}$ on $T$:
\begin{align*}
	\nu_\beta(A) := \frac{\sum_{t \in A}e^{\beta X_t}}{Z(\beta)}, \qquad A \subseteq T
\end{align*}
where the parameter $\beta$ is the inverse temperature. In the infinite-temperature limit $\beta \to 0$, $\nu_\beta$ converges to the uniform distribution on $T$, which we will denote by $\nu_0$. If we write $\langle f \rangle_\beta$ for the expectation of a function $f : T \to \Reals$ w.r.t.\ $\nu_\beta$, then we can express the quantities \eqref{eq:quenched_Gibbs_energy} and \eqref{eq:quenched_free_energy} as
\begin{align*}
	g(\beta) = \Ex[\langle X \rangle_\beta] \qquad \text{and} \qquad \phi(\beta) = \Ex\Big[\langle X \rangle_\beta - \beta^{-1} D(\nu_\beta\|\nu_0)\Big],
\end{align*}
where $D(\cdot\|\cdot)$ is the relative entropy (or Kullback--Leibler divergence). The term ``quenched'' refers to the fact that the disorder $(X_t)_{t \in T}$ is frozen (or quenched) when computing the Gibbs averages $\langle \cdot \rangle_\beta$, and the expectation $\Ex[\cdot]$ w.r.t.\ the randomness of $X$ is taken at the end. 

In light of \eqref{eq:softmax_bounds} and \eqref{eq:gibbs_bounds}, both $g(\beta)$ and $\phi(\beta)$ can be viewed as smoothed versions of the expected maximum of $X$. In particular, both converge to it in the zero-temperature limit $\beta \to \infty$. Since $\sqrt{\log |T|}$ is present on both sides of \eqref{eq:expected_max}, it is natural to ask which quantity plays the role of $\sqrt{\log |T|}$ for finite values of $\beta$. Ideally, it should converge to $\sqrt{\log |T|}$ as $\beta \to \infty$ and to $0$ as $\beta \to 0$. In this paper, we show the following:
\begin{enumerate}
	\item For every $\beta \ge 0$, the quenched Gibbs average $g(\beta)$  can  be upper-bounded as
	\begin{align*}
		g(\beta) \lesssim \sigma \sqrt{\Ex D(\nu_\beta \| \nu_0)}
	\end{align*}
	where $\sigma^2$ is the maximum of the variances ${\rm Var}[X_t]$. Moreover, a Sudakov-type lower bound holds in the low-temperature regime: There exists a threshold $\beta_* > 0$ that depends on the covariance structure of $X$, such that
	\begin{align*}
		g(\beta) \gtrsim a\sqrt{\Ex D(\nu_\beta \| \nu_0)} \qquad \text{for } \beta \ge \beta_*
	\end{align*}
	where $a > 0$ is the minimum separation between the elements of $T$. When the $X_t$'s are i.i.d.\ $\cN(0,\sigma^2)$ random variables, the upper and the lower bounds match up to universal constants and hold for \textit{all} values of $\beta$:
	\begin{align*}
		g(\beta) \asymp \sigma \sqrt{\Ex D(\nu_\beta\|\nu_0)}.
	\end{align*}
	\item The quenched free energy $\phi(\beta)$ can be upper-bounded as
	\begin{align*}
		\phi(\beta) \lesssim \sigma \sqrt{\Ex D_{1/2}(\nu_\beta \| \nu_0)},
	\end{align*}
	where $D_{1/2}(\cdot\|\cdot)$ is the R\'enyi divergence of order $1/2$ (see Section~\ref{ssec:info_theory} for the information-theoretic background). When the $X_t$'s are i.i.d.\ $\cN(0,\sigma^2)$ random variables, there is a matching lower bound for all values of $\beta$:
	\begin{align*}
		\phi(\beta) \asymp \sigma \sqrt{\Ex D_{1/2}(\nu_\beta\|\nu_0)}.
	\end{align*}
\end{enumerate}
Since $a > 0$, all the $X_t$'s are almost surely distinct, so the Gibbs measures $\nu_\beta$ converge to a (random) Dirac measure as $\beta \to \infty$. Consequently, the information divergences in the above inequalities converge to $\log |T|$, as expected. Similarly, since $\nu_\beta$ converges to the uniform measure $\nu_0$ as $\beta \to 0$, our inequalities have the right high-temperature behavior as well (wherever applicable). 

\subsection{Related work}

Approximation of the maxima of random processes on finite sets by their smoothed versions is a widely used technique in statistical physics and in probability theory \citep{Bovier2006,Adler2007,Talagrand_MF,Panchenko_SK,Chatterjee_2014}. However, in most of these applications, one is interested mainly in either the expected maximum itself (corresponding to the zero-temperature limit) or in various asymptotic quantities arising in the so-called \textit{thermodynamic limit} (when the size of the index set tends to infinity). To the best of our knowledge, the upper and lower bounds presented in this paper for $\beta > 0$ are new.

In the context of Sudakov minoration, we should mention recent work of \citet{Liu2022,Liu_2023}, where lower bounds are given for quantities of the form
\begin{align}\label{eq:Liu}
	\Ex \log \int_T e^{\langle t,Y \rangle}\mu(\dif t),
\end{align}
where $\mu$ is a probability measure supported on a compact convex set $T \subset \Reals^N$, $\langle \cdot,\cdot \rangle$ is the Euclidean inner product, and $Y$ is a random vector in $\Reals^N$. It is natural to view \eqref{eq:Liu} as a softmax-type relaxation of the expected ``$Y$-width'' of $T$, $\Ex[\max_{t \in T}\langle t,Y\rangle]$. Liu obtains Sudakov-type lower bounds on \eqref{eq:Liu} by first lower-bounding it by the expected supremum of a related process on a certain subset of $T$ and then using a deep convex-geometric result of \citet{Pajor} to lower-bound this expected supremum in terms of the packing numbers of $T$. The dimension reduction approach to Sudakov minoration due to \citet{Mendelson2019} alsy plays a key role here for obtaining dimension-free lower bounds. In contrast to our results, which are specific to Gaussian processes with finite index sets, the bounds of Liu apply to linear processes of the form $X_t = \langle t,Y \rangle$ with compact convex index sets, and $Y$ is not required to be a Gaussian random vector. 

\subsection{Organization of the paper}

The remainder of the paper is structured as follows. Section~\ref{sec:aux} introduces some concepts from information theory and from statistical physics that will be used later on. Upper and lower bounds on the quenched Gibbs average $g(\beta)$ are stated and proved in Section~\ref{sec:gibbs}, followed by the results on the quenched free energy $\phi(\beta)$ in Section~\ref{sec:free_energy}. Section~\ref{sec:REM} presents an application of our results to deriving upper and lower bounds on the quenched pressure in the Random Energy Model (REM), one of the basic models of physical systems with random disorder \citep{Bovier2006}.

\section{Preliminaries and auxiliary results}
\label{sec:aux}

\subsection{Some notions from information theory}
\label{ssec:info_theory}

Let $\cP(T)$ denote the set of all probability measures on $T$. For any two $\mu,\nu \in \cP(T)$ with $\nu$ strictly positive everywhere, the R\'enyi divergence of order $\alpha \in (0,\infty)$, $\alpha \neq 1$, is defined as
\begin{align*}
	D_\alpha(\mu \| \nu) := \frac{1}{\alpha - 1}\log \sum_{t \in T} \mu^\alpha(t)\nu^{1-\alpha}(t).
\end{align*}
The limit of $D_\alpha(\mu \|\nu)$ as $\alpha \to 1$ exists and is equal to the Kullback--Leibler divergence $D(\mu\|\nu)$, see, e.g., \citet{vanErven2014}. The Shannon entropy of $\mu \in \cP(T)$ is
\begin{align*}
	H(\mu) := - \sum_{t \in T} \mu(t)\log \mu(t).
\end{align*}
We will mainly deal with the situation when $\mu = \nu_\beta$ for some $\beta > 0$ and $\nu = \nu_0$, the uniform distribution on $T$. The expressions for the R\'enyi divergences $D_\alpha(\nu_\beta\|\nu_0)$ involve the log-partition function $\Lambda(\beta) := \log Z(\beta)$, whose first and second derivatives are given by 
$\Lambda'(\beta) = \langle X \rangle_\beta$ and $\Lambda''(\beta) = \langle X^2 \rangle_\beta - (\langle X \rangle_\beta)^2$. Since $\Lambda''(\beta) \ge 0$, it follows that $\Lambda(\beta)$ is a convex function of $\beta$. One useful consequence of this is the following:
\begin{lemma}\label{lm:nu_l2} The function $\beta \mapsto \|\nu_\beta\|^2_2 = \sum_{t \in T}\nu_\beta^2(t)$ is nondecreasing.
\end{lemma}
\begin{proof}Using the definition of $\nu_\beta(t)$, we calculate
	\begin{align*}
		\frac{\dif}{\dif \beta}\|\nu_\beta\|^2_2 &= 2 \sum_{t \in T} \nu_\beta(t)\frac{\dif}{\dif \beta}\nu_\beta(t) \\
		&= 2 \sum_{t \in T}\nu^2_\beta(t)\left(X_t  - \Lambda'(\beta)\right) \\
		&= 2 \|\nu_\beta\|^2_2 \left(\frac{1}{\|\nu_\beta\|^2_2}\sum_{t \in T}  X_t \nu^2_\beta(t) - \Lambda'(\beta)\right) \\
		&= 2 \|\nu_\beta\|^2_2 \left(\Lambda'(2\beta) - \Lambda'(\beta)\right) \\
		&\ge 0,
	\end{align*}
	where the inequality follows from the fact that $\Lambda'(\beta)$ is nondecreasing, as $\Lambda(\beta)$ is convex.\end{proof}
\noindent The quantity $\|\nu_\beta\|^2_2$, which is sometimes referred to as the \textit{participation ratio} (see, e.g., Section~5.3 of \citet{MezardMontanari2009}), can be expressed in terms of the partition function as follows:
\begin{align*}
	\|\nu_\beta\|^2_2 &= \frac{1}{Z^2(\beta)}\sum_{t \in T} e^{2\beta X_t} = \frac{Z(2\beta)}{Z^2(\beta)}.
\end{align*}
The negative logarithm of $\|\nu_\beta\|^2_2$ is equal to $H_2(\nu_\beta)$, the R\'enyi entropy of $\nu_\beta$ of order $2$.

With these preliminaries out of the way, we have
\begin{align}\label{eq:Gibbs_KL}
	D(\nu_\beta \| \nu_0) &= \log |T| - H(\nu_\beta) = \log |T| + \beta \langle X \rangle_\beta - \Lambda(\beta)
\end{align}
and
\begin{align}\label{eq:Gibbs_Renyi_alpha}
	D_\alpha(\nu_\beta\|\nu_0) = \log |T| + \frac{1}{\alpha - 1}\big(\Lambda(\alpha\beta) - \alpha\Lambda(\beta)\big).
\end{align}
In particular, specializing \eqref{eq:Gibbs_Renyi_alpha} to $\alpha = 1/2$ gives
\begin{align}\label{eq:Gibbs_Renyi_half}
	D_{1/2}(\nu_\beta \| \nu_0)  = \log |T| + \log \|\nu_{\beta/2} \|^2_2.
\end{align}

\subsection{A replica-symmetric representation of the quenched Gibbs average}

The following lemma relates the quenched Gibbs average $g(\beta)$ to the geometry of $(T,d)$ and to the Gibbs measure at inverse temperature $\beta$:

\begin{lemma}\label{lm:replica} The quenched Gibbs average can be expressed as
	\begin{align}\label{eq:qbeta_general}
		g(\beta) = \frac{\beta}{2}\sum_{s,t \in T} d^2(s,t)\Ex[\nu_\beta(s)\nu_\beta(t)].
	\end{align}
In particular, if the $X_t$'s are i.i.d.\ $\cN(0,\sigma^2)$ random variables, then
\begin{align}\label{eq:qbeta_iid}
	g(\beta) = \beta\sigma^2(1-r(\beta)),
\end{align}
where  $r(\beta) := \Ex\|\nu_\beta\|^2_2$.
\end{lemma}
\begin{proof}Fix a realization $x = (x_t)_{t \in T}$ of $X$. For each $s \in T$, we can view $\nu_\beta(s)$ as a function of $x$:
	\begin{align*}
		\nu_\beta(s) = f_s(x) = \frac{e^{\beta x_s}}{\sum_{t \in T}e^{\beta x_t}},
	\end{align*}
	with
	\begin{align*}
		\frac{\partial f_s}{\partial x_t}(x) &= \beta\left(\frac{e^{\beta x_t}\boldsymbol{1}\{s = t\}}{\sum_{t' \in T}e^{\beta x_{t'}}} - \frac{e^{\beta x_s}e^{\beta x_t}}{\left(\sum_{t' \in T} e^{\beta x_{t'}}\right)^2}\right) \\
		&= \beta\big(\boldsymbol{1}\{s = t\} \nu_\beta(s) - \nu_\beta(s)\nu_\beta(t)\big).
	\end{align*}
Gaussian integration by parts \citep[Lemma~2.2.4]{Adler2007} then gives
\begin{align*}
	\Ex[X_s \nu_\beta(s)] &= \sum_{t \in T} \Ex[X_s X_t] \Ex\Big[\frac{\partial f_s}{\partial x_t}(X)\Big] \\
	&= \beta \sum_{t \in T} \Ex[X_s X_t] \Ex\big[\boldsymbol{1}\{s = t\} \nu_\beta(s) - \nu_\beta(s)\nu_\beta(t)\big].
\end{align*}
Summing over all $s \in T$, we obtain
\begin{align*}
	g(\beta) &= \sum_{s \in T} \Ex[X_s \nu_\beta(s)] \\
	&= \beta \sum_{s,t \in T} \Ex[X_s X_t] \Ex\big[\boldsymbol{1}\{s = t\} \nu_\beta(s) - \nu_\beta(s)\nu_\beta(t)\big] \\
	&= \beta \sum_{s,t \in T}\Big(\Ex[X^2_s]-\Ex[X_sX_t]\Big)\Ex[\nu_\beta(s)\nu_\beta(t)] \\
	&= \frac{\beta}{2} \sum_{s,t \in T} \Ex[(X_s - X_t)^2] \Ex[\nu_\beta(s)\nu_\beta(t)] \\
	&= \frac{\beta}{2} \sum_{s,t \in T}d^2(s,t)\Ex[\nu_\beta(s)\nu_\beta(t)],
\end{align*}
where we have made use of the relation
\begin{align*}
	\sum_{s,t \in T}\Ex[X^2_s-X^2_t]\Ex[\nu_\beta(s)\nu_\beta(t)] = 0
\end{align*}
(which follows by symmetry) and of the definition of $d^2(s,t)$.

When the $X_t$'s are i.i.d.\ $\cN(0,\sigma^2)$ random variables, we have $d^2(s,t) =2\sigma^2\boldsymbol{1}\{s \neq t\}$. Substituting this into \eqref{eq:qbeta_general} and using the fact that $\nu_\beta$ is a probability distribution results in \eqref{eq:qbeta_iid}. Moreover, as can be seen from the proof of Lemma~\ref{lm:nu_l2}, $\frac{\dif}{\dif \beta}\|\nu_\beta\|^2_2$ is an $L^2$ random variable, so $r'(\beta) = \frac{\dif}{\dif \beta}\Ex\|\nu_\beta\|^2_2 \ge 0$, where the interchange of derivative and expectation is justified by the dominated convergence theorem. \end{proof}

Results of this type are common in the literature on statistical physics of disordered systems (see, e.g., Lemma~1.3.11 in \citet{Talagrand_MF} or Lemma~1.1 in \citet{Panchenko_SK} in the context of the Sherrington--Kirkpatrick model), where they are referred to as \textit{replica-symmetric representations}: Let $(\tau^1,\tau^2)$ be a random element of $T \times T$ sampled from the product measure $\nu_\beta \otimes \nu_\beta$; $\tau^1$ and $\tau^2$ are independent replicas of the  system at inverse temperature $\beta$. Then we can write \eqref{eq:qbeta_general} as
	\begin{align*}
		g(\beta) = \Ex\langle X \rangle_\beta = \frac{\beta}{2} \Ex\langle d^2(\tau^1,\tau^2)\rangle_\beta,
	\end{align*}
	where $\langle\cdot\rangle_\beta$ is the expectation w.r.t.\ $\nu_\beta \otimes \nu_\beta$ and where $\Ex[\cdot]$ is the expectation  w.r.t.\ $X$.

\section{Upper and lower bounds on the quenched Gibbs average}
\label{sec:gibbs}

Our analysis of the quenched Gibbs average $g(\beta)$ relies crucially on the fact that it can be expressed as an expectation of $X$ first w.r.t.\ the Gibbs measure $\nu_\beta$ and then w.r.t.\ the randomness of $X$. To obtain an upper bound on $g(\beta)$, we make use of the Gibbs variational principle \citep{Dupuis_Ellis_1997} which, in the context of statistical physics, captures the Legendre--Fenchel duality between the relative entropy and the free energy. For the lower bound, we make fundamental use of the replica-symmetric representation of $g(\beta)$ in Lemma~\ref{lm:replica} and of Sudakov minoration.

\subsection{Upper bound via the Gibbs variational principle}

The Gibbs variational principle states that, for any function $f : T \to \Reals$,
\begin{align*}
	-\log \langle \nu_0, e^{-f} \rangle = \inf_{\mu \in \cP(T)} \left\{ \langle \mu,f \rangle + D(\mu \| \nu_0)\right\},
\end{align*}
where $\langle \mu, f\rangle$ denotes the expected value of $f$ w.r.t.\ $\mu$.

\begin{theorem}\label{thm:Gibbs_upper} The quenched Gibbs average is bounded from above as
\begin{align*}
	g(\beta) \le \sqrt{2\sigma^2 \Ex D(\nu_\beta\|\nu_0)}
\end{align*}
where $\sigma^2$ is the maximum of the variances ${\rm Var}[X_t]$.
\end{theorem}

\begin{proof} Applying the Gibbs variational principle to the (random) function $f(t) := -\lambda X_t$ for some $\lambda > 0$ (to be chosen later) gives
\begin{align*}
	-\log\Bigg(\frac{1}{|T|}\sum_{t \in T} e^{\lambda X_t}\Bigg) \le -\lambda \langle X \rangle_\beta + D(\nu_\beta \| \nu_0).
\end{align*}
Rearranging and taking expectations w.r.t.\ $X$, we obtain $g(\beta) \le \frac{1}{\lambda} \Ex [D(\nu_\beta \| \nu_0)] + \phi(\lambda)$,
where $\phi(\lambda)$ can be estimated using Jensen's inequality and properties of Gaussian random variables:
\begin{align*}
	\phi(\lambda) &= \frac{1}{\lambda}\Ex\log\Bigg(\frac{1}{|T|}\sum_{t \in T}e^{\lambda X_t}\Bigg) \le \frac{1}{\lambda} \log \Bigg(\frac{1}{|T|}\sum_{t \in T}\Ex e^{\lambda X_t}\Bigg) \le \frac{\lambda \sigma^2}{2}.
\end{align*}
Combining the above and optimizing over $\lambda$ yields
\begin{align*}
	g(\beta) \le \inf_{\lambda > 0} \Big\{\frac{1}{\lambda} \Ex D(\nu_\beta \| \nu_0) + \frac{\lambda \sigma^2}{2}\Big\} = \sqrt{2\sigma^2 \Ex D(\nu_\beta\|\nu_0)},
\end{align*}
and the proof is complete.
\end{proof}
\noindent Using the expression \eqref{eq:Gibbs_KL} for $D(\nu_\beta \| \nu_0)$, we immediately obtain the following:

\begin{corollary} $g(\beta) \le \sqrt{2\sigma^2\big(\log |T|-\Ex H(\nu_\beta)\big)}$.
\end{corollary}
Since the $X_t$'s are distinct a.s., the quenched entropy $\Ex H(\nu_\beta)$ vanishes in the zero-temperature limit $\beta \to \infty$, and we recover the well-known Gaussian maximal inequality with the sharp constant:
\begin{align}\label{eq:Gauss_max}
	\Ex\Big[\max_{t \in T}X_t\Big] \le \sqrt{2\sigma^2 \log |T|}.
\end{align}

\subsection{A low-temperature lower bound via Sudakov minoration}

Since the $X_t$'s are $a$-separated, i.e., $d(s,t) \ge a > 0$ for all distinct $s,t \in T$, we can lower-bound the expected maximum of $X$ via Sudakov minoration (see, e.g., Lemma~2.4.2 in \citet{Talagrand_2014}):
\begin{align}\label{eq:Sudakov}
	\Ex\Big[\max_{t \in T}X_t\Big] \ge ca\sqrt{\log |T|},
\end{align}
where $0 < c < 1$ is a universal constant (Lemma~5.5.6 in \citet{MarcusRosen2006} gives $c = \frac{1}{17}$). Our next result is a low-temperature relaxation of Sudakov minoration:

\begin{theorem}\label{thm:Gibbs_lower} There exists an inverse temperature threshold $0 < \beta_* < \infty$ that depends on the covariance structure of $(X_t)_{t \in T}$, such that
	\begin{align*}
		g(\beta) \ge c a \sqrt{\Ex D(\nu_\beta\|\nu_0)} \qquad \text{for all }  \beta \ge \beta_*
	\end{align*}
where $c > 0$ is the same constant as in \eqref{eq:Sudakov}.
\end{theorem}

\begin{proof} Let $\Delta := \max \{d(s,t) : s,t \in T\}$ denote the diameter of $(T,d)$. Then Lemma~\ref{lm:replica} allows us to bound $g(\beta)$ from above and from below as follows:
	\begin{align*}
		\frac{\beta a^2}{2}(1-r(\beta)) \le g(\beta) \le \frac{\beta\Delta^2}{2}(1-r(\beta)),
	\end{align*}
	where $r(\beta) = \Ex\|\nu_\beta\|^2_2$. The function $r(\beta)$ is nondecreasing by Lemma~\ref{lm:nu_l2} and bounded between $r(0) = \frac{1}{|T|}$ and $1$. Therefore, there exists some value $\beta_* > 0$, such that $1-r(\beta) \le \frac{c^2 a^2}{2\Delta^2}$ for all $\beta \ge \beta_*$. This implies that
	\begin{align}\label{eq:low_temp_regime}
		0 \le \frac{g(\beta)}{\beta} \le \frac{c^2a^2}{4}, \qquad \beta \ge \beta_*.
	\end{align}
Differentiating the function
\begin{align*}
	\tilde{g}(\beta) := g(\beta) - ca \sqrt{\Ex D(\nu_\beta\|\nu_0)} = g(\beta) - ca\sqrt{\beta(g(\beta)-\phi(\beta))}
\end{align*}
w.r.t.\ $\beta$ gives
\begin{align*}
	\tilde{g}'(\beta) 
	&= g'(\beta) - \frac{ca}{2}\frac{\frac{\dif}{\dif \beta} [\beta(g(\beta)-\phi(\beta))]}{\sqrt{\beta(g(\beta)-\phi(\beta))}} \\
	&= g'(\beta)\Bigg(1 - \frac{ca}{2}\frac{\beta}{\sqrt{\beta(g(\beta)-\phi(\beta))}}\Bigg) \\
	&\le g'(\beta)\Bigg(1 - \frac{ca}{2}\sqrt{\frac{\beta}{g(\beta)}} \Bigg) \\
	&\le 0,
\end{align*}
where we have used \eqref{eq:low_temp_regime} and the fact that $g'(\beta) \ge 0$, which follows from the convexity of $\beta \mapsto \beta \phi(\beta)$. Thus, for any $\beta \ge \beta_*$,
\begin{align*}
	\tilde{g}(\beta) \ge \lim_{\beta \to \infty}\tilde{g}(\beta) = \Ex\Big[\max_{t \in T}X_t\Big] - ca\sqrt{\log |T|} \ge 0,
\end{align*}
where the inequality is by Sudakov minoration.
\end{proof}

\subsection{A lower bound for i.i.d.\ processes}

When the $X_t$'s are i.i.d.\ $\cN(0,\sigma^2)$ random variables, the separation condition holds with $a^2 = 2\sigma^2$. In this setting, we can extend the range of our Sudakov-type lower bound on $g(\beta)$ to \textit{all} values of $\beta$, but with a slightly worse constant:

\begin{theorem}\label{thm:Gibbs_lower_iid} When $X_t$'s are i.i.d.\ $\cN(0,\sigma^2)$ random variables, we have
	\begin{align*}
		g(\beta) \ge c\sigma \sqrt{\Ex D(\nu_\beta \| \nu_0)} \qquad \text{for all } \beta \ge 0.
	\end{align*}
\end{theorem}
\begin{remark} As will become evident from the proof, the lower bound with the worse constant $c/2$ holds only in the high-temperature regime $\beta \le \beta_*$, where the threshold $\beta_*$ is determined in exactly the same manner as in Theorem~\ref{thm:Gibbs_lower}. In the low-temperature regime $\beta > \beta_*$, the constant is still $c$. \end{remark}

\begin{proof} In the i.i.d.\ case, $a^2 = \Delta^2 = 2\sigma^2$. Since $r(\beta)$ is nondecreasing, the same choice of $\beta_*$ as in the proof of Theorem~\ref{thm:Gibbs_lower} ensures that
	\begin{align*}
		1 - r(\beta) \ge \frac{c^2}{2}, \qquad \beta \le \beta_*.
	\end{align*}
This is equivalent to the inequality $1-r(\beta) \ge \frac{c}{\sqrt{2}}\sqrt{1-r(\beta)}$ for $\beta \le \beta_*$. For any such $\beta$ we have
\begin{align*}
	g(\beta) = \beta \sigma^2 (1-r(\beta)) \ge c\sigma \sqrt{\frac{1}{2}\beta^2 \sigma^2(1-r(\beta))} = c\sigma \sqrt{\frac{1}{2}\beta g(\beta)} \ge c\sigma \sqrt{\frac{1}{2}\Ex D(\nu_\beta \| \nu_0)},
\end{align*}
where we have used Lemma~\ref{lm:replica} and the fact that $\Ex D(\nu_\beta \| \nu_0) = \beta (g(\beta) - \phi(\beta)) \le \beta g(\beta)$. This establishes the lower bound in the high-temperature regime. The analysis of the low-temperature regime $\beta \ge \beta_*$ is exactly the same as in the proof of Theorem~\ref{thm:Gibbs_lower}.
\end{proof}

\section{Upper and lower bounds on the quenched free energy}
\label{sec:free_energy}

The analysis of the quenched free energy $\phi(\beta)$ is complicated by the fact that, unlike the quenched Gibbs average $g(\beta)$, it cannot be expressed as an expected value of $X$ w.r.t.\ a random measure on $T$. The appropriate strategy here is to capitalize on the easily verified fact that $g(\beta)$ is equal to the derivative of $\beta \phi(\beta)$ and use the fundamental theorem of calculus in conjunction with the already available upper and lower bounds on $g(\beta)$.

\subsection{Upper bound}

The following upper bound is a counterpart of Theorem~\ref{thm:Gibbs_upper} for $g(\beta)$:

\begin{theorem}\label{thm:phi_upper} The quenched free energy is bounded from above by
	\begin{align*}
		\phi(\beta) \le \sqrt{2\sigma^2 \Ex D_{1/2}(\nu_\beta\|\nu_0)},
	\end{align*}
where $D_{1/2}(\cdot\|\cdot)$ is the R\'enyi divergence of order $1/2$.
\end{theorem}

\begin{proof} It is convenient to work with the function $\psi(\beta) := \beta\phi(\beta) = \Lambda(\beta) - \log |T|$, since $g(\beta) = \psi'(\beta)$. Define the shorthand $D(\beta) := \Ex D(\nu_\beta\|\nu_0)$. A straightforward computation gives
	\begin{align}\label{eq:D_vs_psi}
		D(\beta) = \beta \psi'(\beta) - \psi(\beta).
	\end{align}
		Then
	\begin{align*}
		\psi(\beta) &= \int^\beta_0 g(\lambda)\dif \lambda \le \int^\beta_0 \sqrt{2\sigma^2 D(\lambda)}\dif\lambda \le \beta\sqrt{2\sigma^2}\cdot \sqrt{\frac{1}{\beta}\int^\beta_0 D(\lambda)\dif \lambda},
	\end{align*}
	where the first step is by the fundamental theorem of calculus, the second uses Theorem~\ref{thm:Gibbs_upper}, and the third is by Jensen's inequality. Using \eqref{eq:D_vs_psi}, integration by parts, and convexity of $\psi$, we obtain
	\begin{align*}
		\frac{1}{\beta}\int^\beta_0 D(\lambda)\dif \lambda = \frac{1}{\beta}\int^\beta_0 (\lambda \psi'(\lambda) - \psi(\lambda))\dif\lambda = \psi(\beta) - \frac{2}{\beta}\int^\beta_0 \psi(\lambda)\dif\lambda \le \psi(\beta) - 2 \psi\left(\frac{\beta}{2}\right).
	\end{align*}
	Moreover, by the definition of $\psi$, 
	\begin{align*}
		\psi(\beta) - 2 \psi\left(\frac{\beta}{2}\right) &= \log |T| + \Ex \log\Big(\sum_{t \in T}e^{\beta X_t}\Big) - 2\,\Ex \log\Big(\sum_{t \in T}e^{\beta X_t/2}\Big) \\
		&= \log |T| + \Ex \log Z(\beta) - 2\, \Ex \log \sum_{t \in T} (Z(\beta)\nu_\beta(t))^{1/2} \\
		&= \log |T| - 2\, \Ex \log \sum_{t \in T} (\nu_\beta(t))^{1/2}.
	\end{align*}
Recognizing the latter expression as the expectation of the R\'enyi divergence $D_{1/2}(\nu_\beta\|\nu_0)$ in \eqref{eq:Gibbs_Renyi_half} and using this in our bound for $\psi(\beta)$, we obtain the desired estimate.
\end{proof}

\subsection{A lower bound for i.i.d.\ processes}

When the $X_t$'s are $\cN(0,\sigma^2)$ random variables, the inequality in Theorem~\ref{thm:phi_upper} can be reversed:

\begin{theorem} When $X_t$'s are i.i.d.\ $\cN(0,\sigma^2)$ random variables, we have
	\begin{align*}
		\phi(\beta) \ge \frac{c\sigma}{2} \sqrt{\Ex D_{1/2}(\nu_\beta \| \nu_0)} \qquad \text{for all } \beta \ge 0
	\end{align*}
	where $c$ is the constant in the Sudakov minoration bound \eqref{eq:Sudakov}.
\end{theorem}

\begin{proof} We start by applying the fundamental theorem of calculus, but then invoke Lemma~\ref{lm:replica}:
	\begin{align*}
		\psi(\beta) &= \int^\beta_0 g(\lambda) \dif\lambda = \sigma^2 \int^\beta \lambda(1-r(\lambda)) \dif \lambda = \sigma^2 \Big(\frac{\beta^2}{2} - \int^\beta_0 \lambda r(\lambda)\dif\lambda\Big).
	\end{align*}
Integrating by parts and using the monotonicity of $r(\lambda)$, we obtain
\begin{align*}
	\int^\beta_0 \lambda r(\lambda)\dif \lambda = \frac{\beta^2}{2}r(\beta) - \int^\beta_0 \frac{\lambda^2}{2} r'(\lambda)\dif \lambda \le \frac{\beta^2}{2}r(\beta).
\end{align*}
Substituting this into our expression for $\psi(\beta)$ and using Theorem~\ref{thm:Gibbs_lower_iid} gives
\begin{align*}
	\psi(\beta) \ge \frac{\beta^2 \sigma^2}{2}(1-r(\beta)) = \frac{\beta g(\beta)}{2} \ge \frac{c\beta \sigma}{2} \sqrt{\Ex D(\nu_\beta\|\nu_0)},
\end{align*}
where $c$ is the constant in \eqref{eq:Sudakov}. Since the R\'enyi divergence $D_\alpha(\cdot\|\cdot)$ is an increasing function of $\alpha$ \citep{vanErven2014}, we obtain the claimed lower bound.
\end{proof}
\noindent The appearance of the R\'enyi divergence in upper and lower bounds on the quenched free energy is rather natural from the viewpoint of statistical physics, where R\'enyi entropy is used to quantify the sensitivity of equilibrium free energy to changes in temperature \citep{Baez_2022}.

\subsection{A Sudakov-type lower bound without independence}

When the $X_t$'s are no longer independent, it is still possible obtain a Sudakov-type lower bound on a quantity closely related to the quenched free energy $\phi(\beta)$. To state it, we first need to introduce additional notation. For a centered Gaussian process $X = (X_t)_{t \in T}$ and for a subset $A \subseteq T$, let $\Phi_\beta(X; A)$ denote the $\beta$-softmax of $(X_t)_{t \in A}$:
\begin{align*}
	\Phi_\beta(X; A) := \frac{1}{\beta}\log \Big(\sum_{t \in A}e^{\beta X_t}\Big).
\end{align*}
It is easy to see that $\Phi_\beta(X;\cdot)$ is monotone w.r.t.\ set inclusion: if $A \subseteq A'$, then $\Phi_\beta(X; A) \le \Phi_\beta(X; A')$.  We then have the following inequality, which can be viewed as a ``softmax'' relaxation of a result due to \citet{Talagrand1992}:

\begin{theorem}\label{thm:super_Sudakov} Let $S$ be a $4\sigma$-packing of $(T,d)$, i.e., $d(s,t) \ge 4\sigma$ for all $s,t \in S$ with $s \neq t$. Then
\begin{align}\label{eq:soft_super_Sudakov}
	\Ex \Phi_\beta\Big(X; \bigcup_{s \in S}B(s,\sigma)\Big) \ge  \sigma \Ex \Phi_{\beta\sigma}(G; S) + \frac{1}{|S|}\sum_{s \in S} \Ex \Phi_\beta(X; B(s,\sigma)),
\end{align}
where $G = (G_t)_{t \in T}$ consists of independent $\cN(0,1)$ random variables and where $B(s,\sigma) := \{t \in T: d(s,t) \le \sigma \}$ is the closed ball of radius $\sigma$ centered at $s$.
\end{theorem}
\begin{proof} We adapt the proof strategy from \citet[Lemma~5.5.7]{MarcusRosen2006}. Since $S$ is a $4\sigma$-packing subset of $T$, the balls $B(s,\sigma)$ centered on the points $s \in S$ are disjoint. Let  $X^{(s)} = (X^{(s)}_t)_{t \in T}$ be independent copies of $(X_t)_{t \in T}$, one for each $s \in S$. Also, let $G = (G_t)_{t \in T}$ be a collection of  $\cN(0,1)$ random variables independent of everything else. Each $t \in T' := \cup_{s \in S}B(s,\sigma)$ belongs to exactly one of the balls $B(s,\sigma)$; for such $t$, define $Z_t := X^{(s)}_t - X^{(s)}_s + \sigma G_s$. If $t,t'$ are in the same ball $B(s,\sigma)$, then $Z_t - Z_{t'}$ is equal in law to $X_t - X_{t'}$, so $\Ex|Z_t - Z_{t'}|^2 = \Ex|X_t - X_{t'}|^2$. If $t$ and $t'$ belong to disjoint balls, then
\begin{align*}
	\Ex|Z_t - Z_{t'}|^2 &= \Ex|X^{(s)}_t-X^{(s)}_s|^2 + \Ex|X^{(s')}_{t'} - X^{(s')}_{s'}|^2 + \sigma^2 \Ex|G_s - G_{s'}|^2 
\end{align*}
for some $s,s' \in S$ with $s \neq s'$. Then $\Ex |Z_t - Z_{t'}|^2 \le 4\sigma^2$, but, since $t,t'$ are in disjoint balls, we also have $\Ex |X_t - X_{t'}|^2 \ge 4\sigma^2$. Consequently, the centered Gaussian process $Z = (Z_t)_{t \in T'}$ satisfies
\begin{align*}
	\Ex|Z_t - Z_{t'}|^2 \le \Ex|X_t - X_{t'}|^2, \qquad \forall t,t' \in T'.
\end{align*}
Then a Gaussian interpolation argument due to Chatterjee (see, e.g., the proof of Theorem~2.2.5 in \citet{Adler2007}) shows that $\Ex \Phi_\beta(X; T') \ge \Ex \Phi_\beta(Z; T')$. We now proceed to lower-bound $\Ex \Phi_\beta(Z; T')$. Define the random Gibbs measure $\tilde{\nu}(s) := \frac{e^{\beta \sigma G_s}}{\sum_{s' \in S}e^{\beta \sigma G_{s'}}}$ on the set $S$. By construction of $Z$, we write
\begin{align*}
	\Phi_\beta(Z; T') &= \frac{1}{\beta}\log \Big(\sum_{t \in T'} e^{\beta Z_t}\Big) \\
	&= \frac{1}{\beta} \log \Big(\sum_{s \in S}e^{\beta \sigma G_s}\sum_{t \in B(s,\sigma)} e^{\beta(X^{(s)}_t-X^{(s)}_s)}\Big) \\
	&= \sigma\Phi_{\beta\sigma}(G; S) + \frac{1}{\beta} \log \Bigg(\sum_{s \in S} \tilde{\nu}(s)  \sum_{t \in B(s,\sigma)}e^{\beta(X^{(s)}_t - X^{(s)}_s)} \Bigg) \\
	&\ge \sigma\Phi_{\beta\sigma}(G;S) +\sum_{s \in S}\tilde{\nu}(s) \Phi_\beta(X^{(s)}; B(s,\sigma)) - \sum_{s \in S}\tilde{\nu}(s) X^{(s)}_s,
\end{align*} 
where the last step is by Jensen's inequality. Since $G$ is independent of everything else, using the law of iterated expectation gives
\begin{align*}
	\Ex \Phi_\beta(Z; T') &= \Ex[\Ex[\Phi_\beta(Z; T')|G]]  \\
	&\ge \sigma \Ex \Phi_{\beta\sigma}(G;S) + \sum_{s \in S} \Ex[\tilde{\nu}(s)] \Ex[\Phi_\beta(X^{(s)};B(s,\sigma))] \\
	&= \sigma \Ex \Phi_{\beta\sigma}(G;S) + \frac{1}{|S|}\sum_{s \in S}  \Ex\Phi_\beta(X;B(s,\sigma)),
\end{align*}
where $\Ex[\tilde{\nu}(s)] = \frac{1}{|S|}$ follows by symmetry and where we have used the fact that the processes $X^{(s)}$ and $X$ have the same probability law. \end{proof}

Let us now examine some consequences of Theorem~\ref{thm:super_Sudakov}. First, since $\Phi_\beta(X; \cdot)$ is monotone w.r.t.\ set inclusion, it follows that the right-hand side of \eqref{eq:soft_super_Sudakov} is automatically a lower bound for $\Ex \Phi_\beta(X; T)$. Second, we recover Talagrand's minoration result (for finite $T$) in the zero-temperature limit. Indeed, using \eqref{eq:softmax_bounds} and Sudakov minoration in \eqref{eq:soft_super_Sudakov} gives
\begin{align*}
	\Ex\Big[\max_{t \in T}X_t\Big] & \ge \sigma \Ex\Big[\max_{s \in S}G_s\Big] + \min_{s \in S} \Ex\Big[\max_{t \in B(s,\sigma)}X_t\Big] - \frac{\log |T|}{\beta}\\
	&\gtrsim \sigma\sqrt{\log |S|} +  \min_{s \in S} \Ex\Big[\max_{t \in B(s,\sigma)}X_t\Big] - \frac{\log |T|}{\beta}.
\end{align*}
Taking $\beta \to \infty$, we obtain Talagrand's result. Finally, since $\phi(\beta) = \Ex \Phi_\beta(X;T) - \frac{\log |T|}{\beta}$, Theorem~\ref{thm:super_Sudakov} can be used to lower-bound the quenched free energy without assuming independence.

\section{Application: the quenched pressure in the random energy model}
\label{sec:REM}

We now illustrate the use of our results in the context of the Random Energy Model (REM), introduced by \citet{Derrida1980}. The REM is the simplest model of a physical system with random disorder; yet, despite this apparent simplicity, its behavior exhibits many nontrivial features, including a phase transition (see Chapter 5 of \citet{MezardMontanari2009} or Chapter 9 of \citet{Bovier2006} for more details on the REM).

The state space of the REM is the binary cube $T = \{-1,+1\}^N$. Each of the points $t = (t_1,\dots,t_N) \in \{-1,+1\}^N$ corresponds to a configuration of $N$ spins, each of which can point up $(+)$ or down $(-)$. The energies of the configurations $t \in T$ are i.i.d.\  $\cN(0,N/2)$ random variables $X_t$, where the independence  translates into the absence of interactions between the $N$ spins.  One of the quantities of interest is the \textit{quenched pressure}
\begin{align*}
	P_N(\beta) := \frac{1}{N}\Ex \log Z_N(\beta),
\end{align*}
where
\begin{align*}
	Z_N(\beta) = \sum_{t \in \{-1,+1\}^N} e^{\beta X_t}
\end{align*}
is the (random) partition function. The REM exhibits a phase transition in the infinite system limit $N \to \infty$: For $\beta_c = 2\sqrt{\log 2}$,
\begin{align*}
	\lim_{N \to \infty} P_N(\beta) = \begin{cases}
	\log 2 + \frac{\beta^2}{4}, & \beta < \beta_c \\
	\beta\sqrt{\log 2}, & \beta \ge \beta_c
	\end{cases}
\end{align*}
(see \citet{GiardinaStarr2007} for a self-contained proof of this using variational methods). The phase transition manifests itself in the different behavior of $\lim_{N \to \infty} P_N(\beta)$ as a function of $\beta$: quadratic for $\beta \le \beta_c$ and linear for $\beta \ge \beta_c$. We now show that Theorems~\ref{thm:Gibbs_upper} and \ref{thm:Gibbs_lower} can be used to give an elementary derivation of upper and lower bounds on the quenched pressure for each finite $N$. While the constants in these bounds are not sharp, the bounds still pinpoint the different scaling of $P_N(\beta)$ with $\beta$ (quadratic vs.\ linear) in high- and low-temperature regimes.  

\begin{theorem} For each $N$ and each $\beta_0 \ge 0$, define the functions
	\begin{align*}
		\underline{\cQ}_N(\beta; \beta_0) := \begin{cases}
		\log 2 + \frac{c^2\beta^2}{8}, & \beta \le \beta_0 \\
		\log 2 + \frac{c^2\beta^2_0}{8} + c(\beta-\beta_0) \sqrt{\frac{1}{2N}\Ex D(\nu_{N,\beta_0}\|\nu_{N,0})}, & \beta \ge \beta_0
		\end{cases}
	\end{align*}
	and
	\begin{align*}
		\overline{\cQ}_N(\beta; \beta_0) := \begin{cases}
		\log 2 + \frac{\beta^2}{4}, & \beta \le \beta_0 \\
		\log 2 + \frac{\beta^2_0}{4} + (\beta - \beta_0) \sqrt{\frac{1}{N}\Ex D(\nu_{N,\beta}\|\nu_{N,0})}, & \beta \ge \beta_0
	\end{cases}
	\end{align*}
	where $c$ is the constant in the Sudakov minoration bound \eqref{eq:Sudakov} and where $\nu_{N,\beta}(t) = e^{\beta X_t}/Z_N(\beta)$ is the REM Gibbs measure at inverse temperature $\beta$. Then there exists an inverse temperature threshold $\beta_* = \beta_{N,*} > 0$, such that
	\begin{align*}
		\underline{\cQ}_N(\beta; \beta_*) \le P_N(\beta) \le \inf_{\beta_0 \ge 0} \overline{\cQ}_N(\beta; \beta_0).
	\end{align*}
\end{theorem}
\begin{remark} Notice that the quadratic (high-temperature) and the linear (low-temperature) pieces of $\underline{\cQ}_N(\beta;\beta_0)$ agree when $\beta = \beta_0$. The same applies to $\overline{\cQ}_N(\beta;\beta_0)$, although its low-temperature portion contains a nonlinear factor $\sqrt{\Ex D(\nu_{N,\beta}\|\nu_{N,0})} = \sqrt{N\log 2 - \Ex H(\nu_{N,\beta})}$, which can be further upper-bounded by $\sqrt{N \log 2}$. In particular, $\overline{\cQ}_N(\beta; \beta_c)$ is equal to $\log 2 + \frac{\beta^2}{4}$ for $\beta \le \beta_c$ and bounded from above by $\beta\beta_c = \beta\sqrt{\log 2}$ for $\beta \ge \beta_c$. This gives the correct upper bound for every finite $N$ \citep{GiardinaStarr2007}. See also the work of \citet{Olivieri1984} for finite-$N$ corrections to the asymptotic value of the free energy in the REM.
\end{remark}

\begin{proof} Let $g_N(\beta) : = \Ex[\langle X, \nu_{N,\beta}\rangle]$. Then
	\begin{align}\label{eq:pressure_integral}
		P_N(\beta) = P_N(\beta_0) + \frac{1}{N}\int^\beta_{\beta_0} g_N(\lambda)\dif\lambda, \qquad \text{for all } \beta \ge \beta_0 \ge 0.
	\end{align}
Using Jensen's inequality and the properties of Gaussian random variables, we obtain
\begin{align}\label{eq:P_N_quadratic}
	P_N(\beta) \le \log 2 + \frac{\beta^2}{4}.
\end{align}
On the other hand, from \eqref{eq:pressure_integral} and Theorem~\ref{thm:Gibbs_upper} it follows that
\begin{align*}
	P_N(\beta) &\le P_N(\beta_0) + \int^\beta_{\beta_0} \sqrt{\frac{1}{N} \Ex D(\nu_{N,\lambda}\|\nu_{N,0})}\dif \lambda \nonumber\\
	&\le \log 2 + \frac{\beta^2_0}{4} + (\beta - \beta_0) \sqrt{\frac{1}{N}\Ex D(\nu_{N,\beta}\|\nu_{N,0})},
\end{align*}
where the second step follows from \eqref{eq:P_N_quadratic} and from the fact that $\lambda \mapsto \Ex D(\nu_{N,\lambda} \| \nu_{N,0})$ is nondecreasing. This gives  $P_N(\beta) \le \overline{\cQ}_N(\beta; \beta_0)$, which we can tighten further by minimizing w.r.t.\ $\beta_0$.

To obtain the lower bound, let $\beta_{*} > 0$ be such that $r_N(\beta) = \Ex\|\nu_{N,\beta}\|^2_2 \le 1 - \frac{c^2}{2}$ for all $\beta \le \beta_*$. Then, for $\beta \le \beta_*$, using \eqref{eq:pressure_integral} with $\beta_0 = 0$ and Lemma~\ref{lm:replica} we have
\begin{align*}
	P_N(\beta) &= \log 2 + \frac{1}{2}\int^\beta_0 \lambda (1-r_N(\lambda))\dif \lambda \\
	&= \log 2 + \frac{\beta^2}{4} - \frac{1}{2}\int^\beta_0 \lambda r_N(\lambda)\dif \lambda \\
	& \ge \log 2 + \frac{c^2 \beta^2}{8}.
\end{align*}
For $\beta > \beta_*$, we can use \eqref{eq:pressure_integral} with $\beta_0 = \beta_*$ and the fact that $\beta \mapsto g_N(\beta)$ is increasing to write
\begin{align*}
	P_N(\beta) &\ge P_N(\beta_*) + \frac{1}{N}(\beta-\beta_*) g_N(\beta_*) \\
	&\ge \log 2 + \frac{c^2 \beta^2_*}{8} + \frac{1}{N}(\beta-\beta_*) g_N(\beta_*) .
\end{align*}
Moreover, observing that $\beta_*$ is exactly  the inverse temperature threshold Theorem~\ref{thm:Gibbs_lower} gives for $\nu_{N,\beta}$, we can bound $g_N(\beta_*)$ from below by
\begin{align*}
	g_N(\beta_*) \ge c\sqrt{\frac{N}{2} \Ex D(\nu_{N,\beta_*}\|\nu_{N,0})}. 
\end{align*}
Substituting this into our lower bound for $P_N(\beta)$ for $\beta \ge \beta_*$ and then assembling the two pieces for $\beta \le \beta_*$ and for $\beta \ge \beta_*$, we see that $P_N(\beta) \ge \underline{\cQ}_N(\beta; \beta_*)$.
\end{proof}

\section*{Acknowledgments}
This work was supported by the NSF under awards CCF-2348624 (``Towards a control framework for neural generative modeling'') and CCF-2106358 (``Analysis and Geometry of Neural Dynamical Systems'').


\bibliography{REM}

\end{document}